\documentclass[11pt]{amsart}
\usepackage{amsmath, amsthm, amssymb}
\usepackage{verbatim}
\usepackage{enumitem}
\usepackage{xcolor}
\usepackage{hyperref}
\hypersetup{colorlinks,citecolor=blue,plainpages=false,hypertexnames=false}

\renewcommand{\d}{\partial}
\newcommand{\dbar}{\overline{\partial}}
\newcommand{\ddbar}{\sqrt{-1}\d\overline{\d}}

\newtheorem{thm}{Theorem}

\newtheorem{lem}[thm]{Lemma}

\theoremstyle{definition}

\newtheorem{rem}[thm]{Remark}

\renewcommand{\[}{\begin{equation}}
\renewcommand{\]}{\end{equation}}

\newcommand{\al}{\alpha}
\newcommand{\be}{\beta}

\newcommand{\la}{\lambda}

\newcommand{\ep}{\epsilon}

\newcommand{\Ga}{\Gamma}
\newcommand{\La}{\Lambda}

\newcommand{\vep}{\varepsilon}

\newcommand{\n}{\Vert}

\newcommand{\RR}{\mathbb{R}}
\newcommand{\CC}{\mathbb{C}}

\newcommand{\PP}{\mathbb{P}}

\newcommand{\Rc}{\mathrm{Rc}}
\newcommand{\Ric}{\mathrm{Ric}}

\title[Monge-Amp\`ere equations and uniformisation]{The complex Monge-Amp\`ere equation and an application to uniformisation of surfaces}
\author[V. Datar]{Ved Datar}
\author[V. P. Pingali]{Vamsi Pritham Pingali}
\author[H. Seshadri]{Harish Seshadri}

\address{Department of Mathematics, Indian Institute of Science, Bengaluru, India}
\email{vvdatar@iisc.ac.in}
\email{harish@iisc.ac.in}
\email{vamsipingali@iisc.ac.in}
\thanks{The first author (Datar) is supported in part by ANRF MATRICS grant MTR/2022/000260 and an INSA Young Associate fellowship.}

\begin{document}

\maketitle

\begin{abstract}
We prove that a complete noncompact  K\"ahler surface with positive and bounded sectional curvature is biholomorphic to $\CC^2$. This result confirms a special case of Yau's conjecture that a complete noncompact K\"ahler $n$-manifold with positive holomorphic bisectional curvature is biholomorphic to $\CC^n$. In contrast to all known results on Yau's  conjecture, we do not need additional assumptions on the global/asymptotic  geometry of the K\"ahler surface apart from completeness.
Towards this end, we prove that the integral of the square of the Ricci form of a complete K\"ahler surface with positive sectional curvature is finite. The work of Chen and Zhu shows that this latter result implies that the surface is biholomorphic to $\CC^2$ . The main new idea is the construction of a Lipschitz continuous plurisubharmonic weight function with finite Monge-Amp\`ere mass. This weight function is obtained by solving a complex Monge-Amp\`ere equation.

\end{abstract}

\section{Introduction}

Let $(X, \omega)$ be a complete K\"ahler manifold of complex dimension $n$ and positive holomorphic bisectional curvature (which we denote by $BK_\omega>0$). If $X$ is compact, it is biholomorphic to $\PP^n$ by the resolution of the Frankel conjecture by Siu-Yau \cite{SY} and Mori \cite{Mo}. If $X$ is noncompact, a longstanding conjecture of Yau predicts that $X$ must be biholomorphic to $\CC^n$ \cite{Y}. The corresponding statement under the stronger assumption of positive sectional curvature was previously conjectured by Green and Wu \cite{GW2}.

Yau's conjecture has been settled under additional hypotheses, typically involving volume growth and curvature decay / finiteness of certain curvature integrals \cite{GW, MSY, MY, M, Mok89, Mok90, CT}. In \cite{L1}, Liu proved that if $X$ is as above and has maximal volume growth, then it is indeed biholomorphic to $\CC^n$. Subsequently Lee and Tam \cite{LT}, building on earlier work by Chau, Tam and others (cf. \cite{CT} and references therein), proved the same result using the K\"ahler-Ricci flow. Interestingly, the other extreme case of {\it minimal} volume growth has also been settled, at least in complex dimension $2$: In \cite{CZ1}, Chen-Zhu proved that if $X$ is a complete noncompact Kahler manifold with $BK>0$ and $p \in X$, then there exists $C>0$ such that $$Vol(B(p,r)) \ge Cr^n,$$ where $n =dim_{\CC}X$.  In a later work \cite{CZ4}, they proved that if $X$ has positive {\it sectional} curvature and minimal volume growth, then it is biholomorphic to an affine algebraic variety. If $dim_{\CC}X=2$, it follows that $X$ is biholomorphic to $\CC^2$ by a classical result of Ramanujam \cite{R}. In a different direction, Chen-Zhu \cite{CZ2} proved that if $X$ has positive and {\em bounded} sectional curvature, and satisfies $$\int_X \Rc_\omega^n < \infty,$$ then $X$ is biholomorphic to a quasi-projective variety (cf. \cite{Mok90, To}). Again, if $\dim_{\CC}X =2$, $M$ is biholomorphic to $\CC^2$.

 Yet another conjecture of Yau \cite{Yau2}, refined by Yang \cite{Yang}, states that if $(X^n,\omega)$ has non-negative bisectional curvature and $o\in X$ is a given point, then there exists a constant $C>0$ such that for any $r>0$ and $k=1,\cdots,n$, $$r^{2k-2n}\int_{B(o,r)} \Rc_\omega^k\wedge \omega^{n-k}\leq C. $$
These are higher-dimensional versions of the classical Cohn-Vossen inequality \cite{CV}. If $X^n$ has positive sectional curvature, then the $k=1$ case of the conjecture follows from the following a priori estimate of Petrunin \cite{Pet} and rescaling: There exists a dimensional constant $C(n)$ such that  for any $o\in X$, the scalar curvature $S_\omega$ satisfies $$\int_{B(o,1)}S_\omega \omega^n < C.$$ Using our weight function and the method of \cite{CZ2} we can establish the remaining case of Yau's Cohn-Vossen type conjecture for K\"ahler surfaces with positive sectional curvature.
\begin{thm}\label{thm:chern-weil}
Let $(X, \omega)$ be a complete noncompact K\"ahler surface with positive sectional curvature. Then $$\int_X \Rc_\omega^2 < \infty.$$
\end{thm}

Combining Theorem \ref{thm:chern-weil}  with the result of Chen and Zhu mentioned above \cite{CZ2} we obtain the following special case of the conjectures of Green-Wu and Yau.

\begin{thm} \label{thm:sec-bounded}
Let $(X, \omega)$ be a complete noncompact K\"ahler surface with positive and bounded sectional curvature. Then $X$ is biholomorphic to $\CC^2$.
\end{thm}

Note that the positive sectional curvature hypothesis implies that $X$ is diffeomorphic to $\RR^4$, by  the Gromoll-Meyer theorem \cite{GM}. It also implies that $X$ is a Stein manifold equipped with a smooth strictly plurisubharmonic (psh) exhaustion function that is uniformly Lipschitz, by the work of Green-Wu \cite{GW}. Neither of these facts is known to hold under the weaker assumption $BK_\omega>0$.
\vspace{2mm}
%As an immediate consequence, we have the following result.

%\begin{cor}
%Let
%\end{cor}

The key new technical input in this paper, which may be of independent interest, is the following construction of a uniformly Lipschitz psh weight function with finite Monge-Amp\`ere mass.

\begin{thm}\label{thm:weight}
Let $(X^n,\omega)$ be a complete, non-compact $n$-dimensional K\"ahler manifold with a smooth exhaustion function $\rho:X\rightarrow \RR$ with uniformly bounded gradient. Suppose $BK_\omega >0$. Then there exists a uniformly Lipschitz, strictly plurisubharmonic function $\phi$ on $X$ such that $$\int_X(\ddbar \phi)^n < \infty.$$ In particular, in view of the theorem of Green and Wu \cite{GW}, such a $\phi$ exists on any complete, non-compact K\"ahler manifold with positive sectional curvature.
\end{thm}

Note that there is no dimension restriction in the above theorem. \vspace{2mm}

%The dimension restriction is essential in the proof of Theorem \ref{thm:chern-weil}.

The present work originated from an attempt to use non-smooth weights to resolve Yau’s uniformisation conjecture for Kähler surfaces with positive sectional curvature, without imposing an upper bound. Although that objective remains unfulfilled, a finite stratification result for such manifolds will be presented in a forthcoming paper now in preparation.  \vspace{5mm}

\section{Construction of a weight function with bounded mass}

We prove Theorem \ref{thm:weight} in this section. Without loss of generality, we may also assume that $\rho \geq 0$, and that $|\nabla \rho|_\omega \leq 1$. We fix a point $o\in \rho^{-1}(0)$. Note that in the case of positive sectional curvature, by Green and Wu's work one can choose $\rho$ to be strictly convex, and hence $o$ be the unique point in $\rho^{-1}(0)$. Next, we set $$B_R:= \{x\in M~|~ \rho(x) < R\},$$ and $S_R := \partial B_R$. Since $\rho$ is a strictly psh exhaustion function, $B_R$ is a strictly pseudoconvex compact subset of $X$. Let $R_\nu\rightarrow\infty$ be a sequence of regular values of $\rho$. Set $B_\nu = B_{R_\nu}$ and $S_\nu = \partial B_\nu$. The main idea is to solve a complex Monge-Amp\`ere equation on $B_\nu$ with a rapidly decaying right hand side and take a limit  of the solutions. The key point is to obtain apriori gradient estimates so as to obtain a Lipschitz function in the limit. We first need to construct a rapidly decaying function at infinity that is smaller in some precise sense than the curvature decay at infinity.  For any point $p\in M$ we let $$\La(x) := \inf_{u,v \in T_x M} \frac{\mathrm{Rm}(u,v,v,u) + \mathrm{Rm}(u,Jv,Jv,u)}{|u\wedge v|^2}.$$ By our hypothesis, $\La(p) >0$ for all $p$. We then have the following elementary observation.

\begin{lem}
There exists a smooth exhaustion function $F:[0,\infty) \rightarrow\RR$ and a constant $C>0$ such that the following properties hold:
\begin{enumerate}
\item $$\int_X e^{-F(\rho)/n}\omega^n < \infty,$$
\item $$e^{-F(\rho)}\omega^n < \frac{1}{2}(\ddbar \rho)^n,$$ and
\item For every $x\in X$, $$\n \nabla F(\rho)\n _\omega e^{-F(\rho)/n}(x) \leq \frac{n}{4}\La(x)$$
\end{enumerate}
\end{lem}

\begin{proof}
Fix a point $o\in X$. Let $g:[0,\infty)\rightarrow (0,\infty)$ be a smooth strictly decreasing function satisfying $$g(r) < \Big( \inf_{x\in S_r}\min\Big(e^{-d_\omega(o,x)}, \frac{(\ddbar\rho)^n}{2\omega^n}(x)\Big)\Big)^{1/n}.$$ Note that by the exhaustive nature of $\rho$, $S_r$ is non-empty for each $r\geq 0$. Moreover, the right-hand-side is some continuous, strictly positive function in $r$, and hence this is always possible. One can easily see that $\displaystyle \lim_{r\rightarrow\infty}g(r) = 0.$ We now define $$F(\rho) = -n\ln\Big(c - \int_{0}^\rho G(s)\,ds\Big), $$ where $$c = \int_0^\infty G(s)\,ds,$$ and we choose a smooth, strictly positive function $G$ so that
\begin{itemize}
\item $G(s) < -g'(s)$ for all $s$, and
\item $G(r) < \frac{1}{4}\inf_{x\in S_r}\La(x).$
\end{itemize}
From the first condition it follows that $$\int_r^\infty G(s)\,ds < g(r).$$ In particular $G$ is integrable with the total integral $c< g(0)$. It is now easy to check that $F(\rho)$ satisfies all the required properties.
\end{proof}
\begin{comment}
\begin{proof}
We define $$F(\rho) = -n\ln\Big(1 - \int_{-1}^\rho G(s)\,ds\Big), \rho$$ for a suitable function $G:[-1,\infty)\rightarrow (0,\infty).$ In order for $F$ to have the above properties, one can pick a $G$ so that
\begin{itemize}
\item $$\int_{-1}^\infty G(s)\,ds = 1,$$
\item For some fixed point $o\in X$, $$\Big(\int_{r}^\infty G(s)\, ds\Big)^n <  \inf_{x\in S_r}\min\Big(e^{-d_\omega(o,x)}, \frac{(\ddbar\rho)^n}{2\omega^n}(x)\Big),$$ and
\item $$G(r) < \frac{1}{4}\inf_{x\in S_r}\La(x).$$
\end{itemize}
\end{proof}

\begin{comment}
\begin{proof}
We define $$F(\rho) = -n\ln\Big(1 - \int_{-1}^\rho G(s)\,ds\Big), \rho$$ for a suitable function $G:[-1,\infty)\rightarrow (0,\infty).$ In order for $F$ to have the above properties, one can pick a $G$ so that
\begin{itemize}
\item $$\int_{-1}^\infty G(s)\,ds = 1,$$
\item For some fixed point $o\in X$, $$\Big(\int_{r}^\infty G(s)\, ds\Big)^n <  \inf_{x\in S_r}\min\Big(e^{-d_\omega(o,x)}, \frac{(\ddbar\rho)^n}{2\omega^n}(x)\Big),$$ and
\item $$G(r) < \frac{1}{4}\inf_{x\in S_r}\La(x).$$
\end{itemize}
\end{proof}
\end{comment}

Next, by \cite{BL}, there exists a strictly psh $u_\nu\in C^\infty(B_\nu)\cap C^0(\overline{B_\nu})$ such that $$\begin{cases} (\ddbar u_\nu)^n =  e^{-F(\rho)}\omega^n\\ u_\nu\Big|_{S_\nu} = \rho = R_\nu.\end{cases}$$ By the comparison principle and our choice of $F$ we have that $u_\nu\geq \rho$. The key point is the following:

 \begin{lem}
 For all $\nu$, $$\sup_{B_\nu}|\nabla u_\nu|_\omega^2 \leq1.$$
 \end{lem}
\begin{proof} We first use the maximum principle (cf. Blocki \cite{B}) to reduce the estimate to the boundary. For ease of notation, we will drop the subscript $\nu$.  We also let $\omega' := \ddbar u$ and denote all quantities (such as the Laplacian) associated to $\omega'$ with a prime. Suppose $x_0\in A_R$ is a maxima for $|\nabla u|^2$. We choose normal coordinates for $\omega$ so that $\ddbar u$ is diagonal with entries $(\la_1,\cdots,\la_n).$ We also denote by $h$ the endomorphism of $T^{1,0}X$ given by $h^{i}_k = g^{i\bar j}u_{k\bar j}$. We now compute at $p$,
  \begin{align*}
 \Delta'|\nabla u|^2 &= u^{i\bar j}\partial_i\partial_{\bar j}(g^{k\bar l}u_k u_{\bar l})\\
 &= u^{i\bar j}R_{i\bar j k\bar l}g^{k\bar q}g^{p\bar l}u_k u_{\bar l} + \mathrm{Tr}(h) + u^{i\bar j}g^{k\bar l}u_{ik}u_{\bar j \bar l} + u^{i\bar j}g^{k\bar l}u_{k \bar j;i } u_{\bar l} + u^{i\bar j}g^{k\bar l}u_{k}u_{\bar l\bar j;i}.
  \end{align*}
  Note that the final two terms involve the third derivatives of $u$. One can use the equation to simplify these terms. Indeed, taking log and differentiating the equation with respect to $\partial_k$ we obtain $$u^{i\bar j}u_{i\bar j;k} = -\partial_k F(\rho) + g^{p\bar q}g_{p\bar q;k} = -\partial_k F(\rho)$$ since we are working with normal coordinates for $\omega$ at $p$. We can also similarly write a formula for the $\bar l$-derivative. Note also that $u_{i\bar j;k} = u_{k\bar j;i}$ and so we get that $$\Delta'|\nabla u|^2 =  u^{i\bar j}R_{i\bar j k\bar l}g^{k\bar q}g^{p\bar l}u_k u_{\bar l} + \mathrm{Tr}(h) + u^{i\bar j}g^{k\bar l}u_{ik}u_{\bar j \bar l} - 2\langle \nabla u,\nabla F(\rho)\rangle_\omega.$$ In normal coordinates the third term takes the forms $$u^{i\bar j}g^{k\bar l}u_{ik}u_{\bar j \bar l} = \sum_{i,k}\frac{|u_{ik}|^2}{\la_i} \geq 0.$$ So finally we have
  \begin{align*}
  \Delta'|\nabla u|^2 &\geq  u^{i\bar j}R_{i\bar j k\bar l}g^{k\bar q}g^{p\bar l}u_k u_{\bar l}  -  2|\nabla F(\rho)|\nabla u|\\
  &\geq \La|\nabla u|^2 \mathrm{Tr}(h^{-1}) - 2|\nabla F(\rho)|\nabla u|\\
  &\geq \Big(n\La e^{F(\rho)/n} - 2\frac{F'(\rho)}{|\nabla u|}\Big)|\nabla u|^2,
  \end{align*}
  where we used the arithmetic-geometric mean inequality in the final line. Without loss of generality we may assume that $|\nabla u|^2(p) \geq 1$. Plugging this in we obtain $$  \Delta'|\nabla u|^2  \geq (n\La e^{F(\rho)/n} - 2F'(\rho))|\nabla u|^2 > 0,$$ which contradicts the maximum principle. To summarise, we have proven that $$\sup_{\overline{B_R}}|\nabla u|^2 = \sup_{\partial B_R}|\nabla u|^2.$$ Clearly we only need to bound the normal derivative at the boundary. In the neighbourhood of the boundary, note that by the comparison principle, $$\rho \leq u_\nu \leq R_\nu.$$ The normal derivative $\nabla_nu = \langle \nabla u , n\rangle$, where $n = -\nabla \rho /|\nabla\rho|$ is the inward pointing normal, then clearly satisfies $$ \nabla_n\rho \leq \nabla_n u \leq 0,$$ and hence we have the required bound.
\end{proof}

\begin{proof}[Proof of Theorem \ref{thm:weight}]
We let $\phi_\nu(x)  = u_\nu(x) - u_\nu(o)$. Then $|\nabla\phi_\nu| \leq 1$. Moreover, since $\phi_\nu(o) = 0$, for on any compact set $K$, there exists a constant $C_K$ such that $\n \phi_\nu\n _{C^1(K)}\leq C_K$. By a standard Arzela-Ascoli and diagonal argument, after passing to a subsequence, $\phi_\nu$ uniformly converge on compact sets to a Lipschitz function $\phi$ with $\mathrm{Lip}(\phi)\leq 1.$ By the continuity of the Monge-Amp\`ere operator under uniform limits (cf. \cite[pg.\ 147]{Dem-book}), $\phi$
 solves the Monge-Amp\`ere equation $$(\ddbar \phi)^n = e^{-F(\rho)}\omega^n,$$ and hence by construction, $$\int_M(\ddbar\phi)^n < \infty.$$

 \end{proof}

\section{Proof of Theorem \ref{thm:chern-weil}}

 \subsection{Smoothening by heat flow} Recall that since $\Rc_\omega \geq 0$, there exists a unique, positive, symmetric and stochastically complete heat kernel $H(x,y,t)$. We let $$u(x,t) = \int_X H(x,y,t)\phi(y)\,dy.$$ Then $u(x,t)$  is a a solution to the heat equation with initial condition $u(x,0) = \phi(x)$. By \cite{NT}, $u(x,t)$ is strictly psh for each $t$. We let $\psi(x):= u(x,1)$. Then by \cite{CZ1}, there exists a constant $A>0$ such that we have the following estimates:
 \begin{align*}
 |\nabla u|,~t|\ddbar u| \leq A.
  \end{align*}
   We will work with both the non-smooth weight $\phi$ and the smooth weights $u(x,t)$. To switch back and forth between the two weights, we need the following crucial estimate. The argument appears to be standard but we found it through Chatgpt 5.0.
\begin{lem}\label{lem:heat-est}
There exists a dimensional constant $c(n)$ such that for any $0\leq t_1\leq t_2$, $$u(x,t_2) \leq u(x,t_1) + Ac\sqrt{t_2-t_1},$$ where $A$ is the Lipschitz constant for $\phi$. In particular, $\psi(x)\leq\phi(x) + cA$.
\end{lem}
\begin{proof}
 We have the representation formula $$u(x,t_2) = \int_XH(x,y,t_2-t_1)u(y,t_1)\,dy.$$From the fact that $u(x,t_1)$ is also Lipschitz with the Lipschitz constant $A$, and stochastic completeness, it follows that $$u(x,t_2) \leq u(x,t_1) + A\int_X H(x,y,t)d(x,y)\,dy,$$ where we set $t = t_2-t_1$.  So it suffices to estimate the integral on the right. By the fundamental Li-Yau gradient estimates \cite{LY}, we have the following Gaussian estimate: $$H(x,y,t) \leq \frac{C}{|B(x,\sqrt{t})|}e^{-c\frac{d^2(x,y)}{t}},$$ for some $c<1/4$. Now for integers $k\geq 0$, consider the annuli $$A_k = B(x,(k+1)\sqrt{t})\setminus B(x,k\sqrt{t}).$$ Then
\begin{align*}
\int_X H(x,y,t)d(x,y)\,dy &\leq \sum_k\int_{A_k}H(x,y,t)d(x,y)\,dy\\
&\leq \frac{C\sqrt{t}}{|B(x,\sqrt{t})|} \sum_k (k+1) e^{-ck^2}|A_k|.
\end{align*}
But now by the Bishop-Gromov inequality, $$\frac{|A_k|}{|B(x,\sqrt{t})|} \leq \frac{|B(x,(k+1)\sqrt{t})|}{|B(x,\sqrt{t})|} \leq \omega_4(k+1)^4,$$ and so $$\int_X H(x,y,t)d(x,y)\,dy \leq C\omega_2\sqrt{t}\sum_{k=1}^\infty (k+1)^{5}e^{-ck^2} \leq c\sqrt{t}.$$

\end{proof}

 \begin{rem}
An interesting question is whether $\psi$ continues to have finite Monge-Amp\`ere mass.
\end{rem}

We also need the following basic observation on constructing suitable cut-off functions.

\begin{lem}\label{lem:cut-off}
Let $(X,\omega)$ satisfy $BK_\omega \geq 0$. Fix $o\in X$. Then there exist $0< \theta < 1$, $A>0$ and $a_0>0$ such that the following holds: for all $a>a_0$ there exist a smooth function $\chi_a:X\rightarrow [0,1]$ having the following properties:
\begin{enumerate}
\item $\chi_a \equiv 1$ on $B(o,\theta a)$ and $\mathrm{Supp}(\chi_a)\subset B(o,\theta^{-1}a).$
\item There exists a constant $A$ such that $$|\nabla \chi_a|, |\ddbar \chi_a| \leq \frac{A}{a}. $$
\end{enumerate}
\end{lem}
\begin{proof}
This is standard, so we only sketch the proof. Let $u(x,t)$ solve the heat equation with $u(x,0) = d(x,o)$, and let $\eta(x) = u(x,1)$. Then there exists a constant $C>1$ such that
\begin{itemize}
\item $C^{-1}(1+ d(x,o))\leq \eta(x) \leq C(1+d(x,o))$.
\item $|\nabla \eta|, |\ddbar\eta| < C.$
\end{itemize}
Now let $\chi:\RR\rightarrow [0,1]$ be the usual cut-off function such that $\chi \equiv 1$ on $t\leq 1$ and $\mathrm{Supp}(\chi) \subset (-\infty, 2]$. Then $$\chi_a(x) := \chi \Big(\frac{\eta(x)}{a}\Big)$$ does the job with $\theta = (2C)^{-1}$ and $a_0 = 2C$.
\end{proof}

\subsection{Estimates on holomorphic sections} We first recall the following classical theorem of Hormander and Andreotti-Vesentini.

\begin{thm}[Hormander, Andreotti-Vesentini]\label{thm:l2-existence} Let $(X,\omega)$ be a complete K\"ahler manifold, let $u$ be a smooth function on $X$, and let $L$ be a holomorphic line bundle equipped with a smooth hermitian metric $h$ such that the curvature satisfies $$\ddbar u + \sqrt{-1}\Theta_h + \Ric(\omega) \geq c(x)\omega,$$ for some continuous function $c:X\rightarrow (0,\infty)$. Suppose we have an $L$-valued $(0,1)$ form $\be$ satisfying $\dbar \be = 0$ and $$\int_X \frac{\n \be\n ^2}{c}e^{-u} \omega^n < \infty.$$ Then there exists a unique $\xi \in \Ga(L)$ satisfying $\dbar \xi = \be$ and the $L^2$-estimate $$\int_X\frac{|\xi|^2}{c}e^{-u}\omega^n \leq \int_X \frac{\n \be\n ^2}{c}e^{-u} \omega^n.$$

\end{thm}

 With $\phi$ and $\psi$ as in the previous section, we consider Hermitian metrics $h_{q\phi} = e^{-q\phi}(\omega^n)^{-1}$ and $h_{q\psi} = e^{-q\psi}(\omega^n)^{-1}$ on the canonical bundle $K_M$. We denote the norms simply as $\n \cdot\n _{q\phi}$ and $\n \cdot \n _{q\psi}$ respectively (suppressing in particular the dependence on $\omega$)

\begin{lem}\label{lem:section-est}
Let $u$ be a solution to the heat equation with initial data $u(x,0) = \phi$ as above. Then there exists a $q>>1$, a non-trivial holomorphic section of $s\in H^0(X,K_X)$ and a constant $C>0$ such that for all $0\leq t\leq 1$, $$\int_X\n s\n _{{qu }}^2\omega^2,\int_X\n \nabla_{qu}s\n _{qu}^2\omega^2 < C.$$
\end{lem}
\begin{proof}
As above, let $\psi = u(x,1)$. The proof of existence of holomorphic sections, $L^2$-integrable with respect to the weight $\psi$ is standard.  Nevertheless, we include an outline for the convenience of the reader. Fix a point $o\in X$. By scaling we may assume that there exist holomorphic coordinates $(z^1,z^2)$ on the ball $B(o,2)$ . Let $\chi$ be a cut-off function with support in $B(o,2)$ such that $\chi \equiv 1$ on $B(o,1)$. Let $\sigma = \chi\cdot dz^1\wedge dz^2$ and $\be = \dbar_{K_M}\sigma$, where $\dbar_{K_M}$ in the $\dbar$ operator on the canonical line bundle. Then $\be$ is a $\dbar$-closed $(1,0)$ $K_M$-valued form. We now apply the H\"ormander-Andreotti-Vesentini Theorem \ref{thm:l2-existence} above with the Hermitian metric $h = (\omega^2)^{-1}$ and the weight function $$\tilde\psi = q\psi + 4\chi\cdot\log|z|^2.$$ For $q>>1$,  clearly $\ddbar\tilde\psi > c\omega$, for some continuous function $c \in C^0(X)$. Without loss of generality we can assume that $c<1$ and that $c>\delta$ on $U$.  Moreover $$\ddbar\tilde\psi + \Theta_h + \Rc(\omega) = \ddbar\tilde\psi >0,$$ and so there exists a solution $\xi\in \Ga(K_X)$ to $\dbar_{K_X}\xi = -\be$ satisfying $$\int_X|\xi|^2e^{-q\psi - 4\chi\log|z|^2}\omega^2 < C.$$ In particular, $\xi(o) = 0$. Clearly, we also have that $$\int_X|\xi|^2e^{-q\psi}\omega^2<\infty.$$Finally let $s = \sigma + \xi$. Then $s\in H^0(M,K_X)$ and satisfies $$\int_X|s|^2e^{-q\psi}\omega^2 < C,$$ for some $C>0$. By the estimate in Lemma \ref{lem:heat-est}, we see that there exists a dimensional constant $A$ such that for all $0\leq t\leq 1$, $$\psi(x) \leq u(x,t) + A\sqrt{1-t}.$$ Then $$\int_X \n s\n_{qu}^2\omega^2 = \int_X|s|^2e^{-qu}\omega^2 \leq e^{Aq}C.$$ To obtain a gradient bound, we make use of the following B\"ochner-Weitzenb\"ock formula: $$\Delta \n s\n _{q\psi}^2 = \n \nabla_{q\psi}s\n _{q\psi}^2 - q\Delta\psi + S_\omega\n s\n _{q\psi}^2 \geq \n \nabla_{q\psi}s\n _{q\psi}^2  - A\n s\n _{q\psi}^2,$$ where $S_\omega$ is the scalar curvature of $\omega$.  Let $\chi_a$ be the family of cut-off functions from Lemma \ref{lem:cut-off}. Then multiplying by $\chi_a^2$ and integrating by parts we obtain
\begin{align*}
\int_X\chi_a^2\n \nabla_{q\psi}s\n _{q\psi}^2\omega^n &\leq A\int_{X}\chi_a^2\n s\n _{q\psi}^2\omega^n + \int_{X}\Delta\chi_a^2\n s\n _{q\psi}^2\omega^n\\
&\leq C'.
\end{align*}

This implies an upper bound for $\int_X \n\nabla_{qu}s\n^2_{q\phi}\omega^2$ for all $t\in [0,1]$: Since
$$ \nabla_{qu}s = \nabla_{q\psi}s + q\,s\,d(u-\psi),$$
and $u, \psi$ are Lipschitz with Lipschitz constant $A$, we get
$$  |\nabla_{qu}s|_{qu}^2 \le 2|\nabla_{q\psi}s|_{qu}^2 + 2q^2A^2\,|s|_{qu}^2.$$
Lemma \ref{lem:heat-est} then once again gives
\begin{align*}
\int_X \n\nabla_{qu}s\n^2_{qu}\omega^n \leq C'' \Big(\int_X\n\nabla_{q\psi}s\n_{q\psi}^2\omega^n + \int_X \n s\n_{q\psi}^n\omega^2 \Big) \leq C''C,
\end{align*}
for some $C''$ depending only on $q$ and $A$.

\end{proof}

As a consequence we obtain the following estimate:

\begin{lem}\label{lem:est-l2norm-can-ddbarphi-omega}
Let $s\in H^0(X,K_X)$ be the holomorphic section constructed in Lemma \ref{lem:section-est} above. Then there exists a constant $C$ such that for all $0\leq t\leq 1$, $$\int_X  \n s \n^{2}_{qu}\ddbar u \wedge\omega\leq C.$$
\end{lem}

\begin{proof}
Let $\chi_a$ be the family of cut-off functions from Lemma \ref{lem:cut-off}, and let  $$I_t(a) := \int_X \chi_a  \n s \n^{2}_{qu}\ddbar u \wedge\omega.$$ The required estimate follows from the following claim: There exists a constant $C>0$ such that  $I_t(a) \leq C$ for all $0\leq t\leq 1$ and for all $a\geq a_0$, where $a_0$ is as in Lemma \ref{lem:cut-off}. First assume that $t>0$. Integrating by parts we see that $$ I_t(a) = \int_X\n s\n^2_{qu}\partial\chi_a\wedge\bar\partial u\wedge\omega + \int_X\chi_a\partial\n s\n_{q u}^2 \wedge\dbar u\wedge\omega.$$ The first term is clearly uniformly bounded by Lemma \ref{lem:section-est} and the uniform Lipschitz control on $u$. For the second term we note that if $\al,\be$ are two $(1,0)$ forms then $|\al\wedge\bar\be\wedge\omega| \leq |\al||\be|\omega^2,$ and since $u$ is Lipschitz with Lipschitz constant bounded by $A$, we then have that $$ \Big|\int_X\chi_a\partial\n s\n_{qu}^2 \wedge\dbar u \wedge\omega\Big| \leq A\int_X\chi_a|\nabla \n s\n^2_{qu}| \omega^2.$$  By the Kato inequality and the AM-GM inequality, for any section $\xi$ of a vector bundle,
\begin{equation}\label{eq:kato}
|\nabla |\xi|^2| \leq 2|\xi||\nabla\xi| \leq |\xi|^2 + |\nabla \xi|^2.
\end{equation}Applying this to $\xi = s$ in the above expression we see that  $$\int_X\chi_a|\nabla \n s\n^2_{qu}| \omega^2\leq \int_X\chi_a\n s\n^2_{qu} \omega^2 + \int_X\chi_a||\nabla_{qu}s||^2_{qu} \omega^2.$$ Each term is uniformly bounded by the Lemma above. Finally, the claimed estimated at $t = 0$ follows from standard Bedford-Taylor theory (cf. \cite{Dem-book}) and the fact that $u(x,t)$ converges uniformly to $\phi(x)$ on compact sets.

\end{proof}

\subsection{Proof of Theorem \ref{thm:chern-weil}}Let $s\in H^0(X,K_X)$ as in Lemma \ref{lem:section-est}.  For $\vep>0$ we consider the $(1,1)$ current $$\zeta_\vep(s) := \ddbar\log ( \n s \n_{q\phi}^2+\vep^2) + q\ddbar\phi. $$

\begin{lem}
$\zeta_\vep(s)$ is a closed, positive $(1,1)$ current satisfying $$\zeta_\vep(s) \geq \frac{\n s\n_{q\phi}^2}{\n s\n_{q\phi}^2 + \vep^2}\Ric(\omega).$$ In particular, by Fatou's lemma, $$\int_X\Ric(\omega)^2 \leq
\liminf_{\vep\rightarrow 0^+} \int_X\zeta_\vep(s)\wedge\zeta_\vep(s),$$ where the wedge product on the right is interpreted in the Bedford-Taylor sense.
\end{lem}

\begin{proof}
Let $\chi$ be a strongly positive $(1,1)$ form on $X$ with compact support.  Without loss of generality, we may assume that $\chi$ is supported on a coordinate neighbourhood $U$. Let $v_\vep = \log(\n s\n^2_{q\phi} + \vep^2).$ By Bedford-Taylor theory,, $$\int_X \zeta_\vep(s)\wedge\chi = \int_X (v_\vep + q\phi)\ddbar\chi.$$ Let $\phi_{\delta}$ be a smoothening of $\phi$ (say via convolution with an approximation to identity) such on the support of $\chi$ we have that $\phi_\delta\rightarrow \phi$ uniformly and in $W^{1,p}$ for all $p>1$. Note also that $\phi_\delta$ is pluri-subharmonic for all $\delta>0$. Let $v_{\vep,\delta}:= \log(\n s\n^2_{q\phi_\delta} + \vep^2),$  and $$\zeta_{\vep,\delta}(s) := \ddbar v_{\vep,\delta} + q\ddbar\phi_\delta$$ in $U$. By the dominated convergence theorem and integration by parts, $$\int_X (v_\vep + q\phi)\ddbar\chi. = \lim_{\delta\rightarrow 0}\int_X \chi\wedge \zeta_{\vep,\delta}(s).$$ Next, for any smooth function $f$ and any constant $c$, an elementary computation shows that on the set where $f\neq 0$, $$\ddbar \log (f+ c) \geq \frac{f}{f+c}\ddbar \log f.$$ Taking $f = \n s\n^2_{q\phi_\delta}$ and $c =\vep^2$ we get that on $X\setminus \{s = 0\}$,
\begin{align*}
\zeta_{\vep,\delta}(s) & \geq  \frac{\n s\n^2_{q\phi_\delta}}{\n s \n^2_{q\phi_\delta} + \vep^2} \ddbar\log||s||^2_{q\phi_\delta} + q\ddbar\phi_\delta \\
&= \frac{\n s\n^2_{q\phi_\delta}}{\n s \n^2_{q\phi_\delta} + \vep^2}\Ric(\omega) + \frac{\vep^2q}{\n s\n^2_{q\phi_\delta} + \vep^2}\ddbar\phi_\delta\\
& \geq  \frac{\n s\n^2_{q\phi_\delta}}{\n s \n^2_{q\phi_\delta} + \vep^2}\Ric(\omega).
\end{align*}
 But then clearly the inequality also holds on all of $X$. Integrating against $\chi$, and letting $\delta\rightarrow 0$ we obtain the required lower bound for $\zeta_\vep(s)$. The second part follows from the Bedford-Taylor theory. Indeed if $S,T$ are closed positive $(1,1)$ currents with bounded local potentials such that $T\geq S$, then $T^2\geq S^2$, where the wedge product is interpreted in the Bedford-Taylor sense. %Next, if $\eta$ is a smooth function with compact support (and once again without loss of generality we may assume that the support is in the same coordinate neighbourhood as that of $\chi$), then  by the Bedford-Taylor theory, $$\int_X\zeta_\vep(s)^2\eta$$

 \end{proof}

Theorem \ref{thm:chern-weil} is now a consequence of the following:

\begin{lem}\label{lem:chern-weil}
For each $k=1,2$, $$\int_X\zeta_\vep(s)^k\wedge (\ddbar\phi)^{2-k} \leq q^k\int_X(\ddbar\phi)^2.$$
\end{lem}

\begin{proof}
 We first prove the Lemma for $k=1$. Let $\chi_a$ be the family of cut-off functions from before (cf. Lemma \ref{lem:cut-off}). It is enough to prove that $$I(a) := \int_X \chi_a^3\ddbar\ln(\n s \n_{q\phi}^2 + \vep^2)\wedge\ddbar\phi \xrightarrow{a\rightarrow \infty}0.$$ From the definition of wedge products in the Bedford-Taylor theory,  we can integrate by parts, and obtain
 \begin{align*}
 I(a)&= \int_X \log\Big(1+ \frac{\n s\n_{q\phi}^2}{\vep^2}\Big) \ddbar\chi_a^3\wedge \ddbar\phi\\
 &\leq  \frac{C}{a}\int_X\log\Big(1+ \frac{\n s\n_{q\phi}^2}{\vep^2}\Big)\omega\wedge\ddbar\phi\\
 &\leq \frac{C}{a\vep^2}\int_X\n s\n^2_{q\phi}\omega\wedge\ddbar\phi\\
 &\leq \frac{C'}{a\vep^2}
 \end{align*} for some constant $C'$ independent of $a$ by Lemma \ref{lem:est-l2norm-can-ddbarphi-omega}.  Note that we used the estimates on $\chi_a$ from Lemma \ref{lem:cut-off} in the second line, and the elementary inequality $\log(1+x) \leq x$ for $x\geq 0$ in the third line.

 Next, we consider $k=2$. It is enough to prove that $$J(a) := \int_X\chi_a^3\ddbar\log( \n s \n_{q\phi}^2+ \ep^2)\wedge\zeta_\ep(s) \xrightarrow{a\rightarrow\infty} 0.$$ Once again by Bedford-Taylor theory, $J(a) = \lim_{t\rightarrow 0^+}J_t(a),$ where
$$J_t(a) = \int_X\chi_a^3\ddbar\log( \n s \n_{q\phi}^2+ \ep^2)\wedge\zeta_{\ep,t}(s), \text{ and }$$ $$\zeta_{\vep,t}(s):= \ddbar\log(\n s\n_{qu}^2 + \vep^2)+ q\ddbar u.$$ It suffices to prove that there exists a constant $C$ (independent of $t$ and $a$) such that $|J_t(a)| \leq Ca^{-1}$. Integrating by parts as before and using the upper bound on $\ddbar\chi_a$ and $|\nabla \chi_a|$, we have
\begin{align*}
|J_t(a)| &\leq \frac{C}{a}\int_{X} \chi_a\log\Big(1 + \frac{ \n s \n_{qu}^2}{\vep^2}\Big) \zeta_{\ep,t}(s)\wedge\omega\\
&=  \frac{C}{a}\int_{X}\chi_a\log\Big(1 + \frac{\n s\n_{qu}^2}{\vep^2}\Big)\ddbar \log\Big(1 + \frac{\n s\n_{qu}^2}{\vep^2}\Big)\wedge\omega \\
&+ \frac{Cq}{a}\int_X\chi_a \log\Big(1 + \frac{\n s\n_{qu}^2}{\vep^2}\Big)\ddbar u\wedge\omega.
\end{align*}
The second integral is uniformly in $t$ and $a$ by Lemma \ref{lem:est-l2norm-can-ddbarphi-omega} and the elementary inequality $\ln(1+x) \leq x$ for $x\geq 0$. For the first integral, we note that it is non-negative and so it is enough to obtain a uniform (in $t$ and $a$) upper bound. Integrating by parts,
\begin{align*}
\text{first integral }&= -\int_{X}\chi_a\sqrt{-1}\partial\log\Big(1 + \frac{\n s\n_{qu}^2}{\vep^2}\Big)\wedge \dbar \log\Big(1 + \frac{\n s\n_{qu}^2}{\vep^2}\Big)\wedge\omega \\
&- \int_X \log\Big(1 + \frac{\n s\n_{qu}^2}{\vep^2}\Big) \sqrt{-1}\partial\chi_a\wedge  \dbar \log\Big(1 + \frac{\n s\n_{qu}^2}{\vep^2}\Big)\wedge\omega\\
&\leq \frac{C}{a}\int_X \log\Big(1 + \frac{\n s\n_{qu}^2}{\vep^2}\Big)  \Big|\nabla \log\Big(1 + \frac{\n s\n_{qu}^2}{\vep^2}\Big)\Big|\omega^2,
\end{align*}
where we used the fact that the first integral in the expression above is non-negative and once again the elementary observation that $|\al\wedge\be\wedge\omega| \leq |\al|\be|\omega^2$ for $(1,0)$ forms $\al,\be$ along with the estimate $|\nabla \chi_a|\leq Ca^{-1}$.  Finally,
\begin{align*}
\int_X \log\Big(1 + \frac{\n s\n_{qu}^2}{\vep^2}\Big)  \Big|\nabla \log\Big(1 + \frac{\n s\n_{qu}^2}{\vep^2}\Big)\Big|\omega^2 &\leq \frac{1}{\vep^2}\int_{X} \frac{\n s\n_{qu}^2}{\n s\n_{qu}^2 + \vep^2}|\nabla \n s\n_{qu}^2| \omega^2\\
&\leq  \frac{1}{\vep^2}\int_{X} |\nabla \n s\n_{qu}^2| \omega^2\\
&\leq  \frac{1}{\vep^2}\int_{X} \n s\n_{qu}\n\nabla_{qu}s\n_{qu} \omega^2\\
&\leq \frac{1}{\vep^2}\Big(\int_X \n s\n_{qu}^2\omega^2\Big)^{1/2}\Big(\int_X \n\nabla_{qu}s\n_{qu}^2\omega^2\Big)^{1/2}\\
&\leq C
\end{align*}
for some uniform constant $C$, independent of $a$ and $t$ by Lemma \ref{lem:section-est}. Note that we once again used the Kato inequality \eqref{eq:kato} in the third line.

\end{proof}

\section*{Acknowledgements}
The second author (Pingali) thanks Zakarias Sj\"ostr\"om Dyrefelt for his kind hospitality in Aarhus, as well as Tamas Darvas and Claude Lebrun for useful email exchanges. The third author (Seshadri) would like to thank Fangyang Zheng for helpful discussions. The authors would also like to thank Gautam Bharali and Purvi Gupta for discussions, support, and encouragement.


\begin{thebibliography}{99}

\bibitem{AV} A. Andreotti and E. Vesentini: \textit{Carleman estimates for the Laplace-Beltrami operator on complex manifolds}, {\sc Publ. Math. Inst. Hautes. Etud. Sci.},\textbf{25} (1965), 81–130.
\bibitem{B} Z. Blocki:  {\em A gradient estimate in the Calabi-Yau theorem}, {\sc Math. Ann.}, {\bf 344} (2009), no. 2, 317--327.
\bibitem{CT} A. Chau and L.-F. Tam: {\em On the complex structure of K\"ahler manifolds with nonnegative curvature}, {\sc J. Differential. Geom.}, {\bf 73} (2006), 491--530.

\bibitem{CG} J. Cheeger and D. Gromoll: {\em On the structure of complete manifolds of non-negative curvature}, {\sc Ann. of Math.}, {\bf 46} (1972), 413--433.
\bibitem{CZ1} B.-L. Chen and X.-P. Zhu: \textit{Volume growth and curvature decay of positively curved
K\"ahler manifolds}, {\sc Quarterly Journal of Pure and Applied Mathematics}, \textbf{1} (20059), 68–108.
\bibitem{CZ2} B.-L. Chen and X.-P. Zhu: \textit{Positively curved complete noncompact K\"ahler manifolds}, {\sc Acta Mathematica Scientia}, \textbf{29} B(4) (2009), 829–845.



 \bibitem{CZ3} B.-L. Chen and X.-P. Zhu: \textit{Yau’s uniformization conjecture for manifolds with
 non-maximal volume growth}, {\sc Acta Mathematica Scientia}, \textbf{38} B(5) (2018), 1468–1484.

\bibitem{CZ4} B.-L. Chen and X.-P. Zhu: \textit{A survey on Yau’s uniformization conjecture},  {\sc Surveys in Diff. Geom.}, \textbf{26} (2021), 13-30.
\bibitem{CV} S. Cohn-Vossen: {\em Kurzeste Wege und totalkrummung auf Flachen}, {\sc Compos.Math.} {\bf 2} (1935), 69--133.
\bibitem{D} J.-P. Demailly: \textit{Mesures de Monge-Amp\`ere et caracterisation g\'eometrique des vari\'et\'es alg\'ebriques affines}, {\sc Mem. Soc. Math. France}, \textbf{19} (1985).

\bibitem{Dem-book} J.-P. Demailly: {\em Complex analytic and differential geometry}, ebook: \url{https://people.math.harvard.edu/~demarco/Math274/Demailly_ComplexAnalyticDiffGeom.pdf}.

\bibitem{GW} R. E. Greene and H. Wu: \textit{On K\"ahler manifolds of positive bisectional curvature
 and a theorem of Hartogs}, {\sc Abh. Math. Semin. Univ. Hambg.}, \textbf{47} (1978), 171–185.

 \bibitem{GW2} R. E. Greene and H. Wu,  {\em Analysis on non–compact K\"ahler manifolds}, {\sc Proc. Symp. Pure. Math.},  Vol 30, Part II, Amer. Math. Soc., (1977).

\bibitem{GM} D. Gromoll and W. Meyer:
\textit{On complete open manifolds of positive
curvature}, {\sc Ann. of Math.}, \textbf{90}(1969), 75-90.

\bibitem{BL} B. Guan and Q. Li, {\em Complex Monge-Ampère equations and totally real submanifolds},  {\sc Adv. Math.}, {\bf 225} (2010), no. 3, 1185--1223.

\bibitem{H} L. H\"ormander: \textit{$L^2$-estimates and existence theorems for the $\dbar$-operator}, {\sc Acta Math.},
\textbf{114} (1965), 89–152.

\bibitem{Hub} A. Huber: {\em On subharmonic functions and differential geometry in the large}, {\sc Comm. Math. Helv.}, {\bf 32} (1957), 13--72.

\bibitem{LT} M.-C. Lee and L.-F. Tam:  {\em Chern–Ricci flows on noncompact complex manifolds}, {\sc J. Differential Geom.}, {\bf 115} (2020), no. 3, 529--564.

\bibitem{LS} P. Li and R. Schoen:
\textit{$L^p$ and mean value properties of subharmonic functions on Riemannian manifolds},  {\sc Acta Math.}, \textbf{153} (1984), 279-301.

\bibitem{LY} P. Li and S.-T. Yau: {\em On the parabolic kernel of the Schrödinger operator}, {\sc Acta Math.}, {\bf 156} (1986), no. 3-4, 153--201.
\bibitem{L1} G. Liu: \textit{On Yau’s uniformization conjecture}, {\sc Cambridge Journal of Mathematics}, \textbf{7}
 (2019),  33–70.

\bibitem{L2} G. Liu: \textit{Gromov-Hausdorff limits of K\"ahler manifolds and the finite generation conjecture},
{\sc Ann. Math.}, \textbf{184} (2016), 775–815.

\bibitem{L3} G. Liu: \textit{Three circle theorem and dimension estimate for holomorphic functions on
 K\"ahler manifolds}, {\sc Duke Math. Journal}, \textbf{165} (2016), 2899–2919.

\bibitem{M} N. Mok: \textit{An embedding theorem of complete K\"ahler manifolds with positive bisectional
 curvature onto affine algebraic varieties}, {\sc Bull. Soc. Math. France}, \textbf{112} (1984), 197–258.

 \bibitem{Mok89} N. Mok:  {\em Compactification of complete Ka\"hler surfaces of finite volume satisfying certain curvature conditions}, {\sc Ann. of Math.}, {\bf 129} (1989), 383--425.

  \bibitem{Mok90} N. Mok:  {\em An embedding theorem of complete Ka\"hler manifolds of positive Ricci curvature onto quasi–projective varieties}, {\sc Math. Ann.}, {\bf 286} (1990), 377--408.

  \bibitem{MSY} N. Mok, S.-T. Siu and S.-T. Yau: {\em The Poincare-Lelong equation on complete K\"ahler manifolds}, {\sc Compos. Math.}, {\bf 44} (1981), 183--218.

  \bibitem{MY} N. Mok and S.-T. Yau: {\em Completeness of K\"ahler-Einstein metrics on bounded domains of holomorphy and the characterization of domains of holomorphy by curvature conditions}, {\sc Proc. Symp. Pure Math.}, {\bf 39}, 41--59.

 \bibitem{MZ} N. Mok and J. Q. Zhong: {\em Compactifying complete Ka\"hler–Einstein manifolds of finite topological type and bounded curvature},  {\sc Ann. of Math.}, {\bf 129} (1989), 427--470.
\bibitem{Mo} S. Mori: \textit{Projective manifolds with ample tangent bundles}, {\sc Ann. of Math.}, \textbf{100} (1979),
 593–606.


\bibitem{NT} L. Ni and L. F. Tam: \textit{Plurisubharmonic functions and the structure of complete
 K\"ahler manifolds with nonnegative curvature}, {\sc J.Diff.Geom.}, \textbf{64} (2003), 457–524.
 \bibitem{Pet} A. M. Petrunin, {\em An upper bound for the curvature integral},  {\sc St. Petersburg Mathematical Journal},  {\bf 20} (2009), no. 2, 255--265.
\bibitem{R} C. P.  Ramanujam:  \textit{A topological characterization of the affine plane as an algebraic
 variety}, {\sc Ann. of Math.}, \textbf{94} (1971), 69–88.
\bibitem{Simha} R. R. Simha:, {\em  On the analyticity of certain singularity sets}, {\sc J. Indian Math. Soc. (N.S.)} {\bf 39} (1975), 281--283.
\bibitem{SY} Y. T.  Siu and S.-T. Yau: \textit{Compact K\"ahler manifolds of positive bisectional curvature},
 {\sc Invent. Math.}, \textbf{59} (1980), 189–204.
 \bibitem{To} W. K. To: {\em Quasi–projective embeddings of noncompact complete K\"ahler manifolds of positive Ricci curvature and satisfying certain topological conditions}, {\sc Duke Math. J.}, {\bf 63} (1991), no. 3, 745--789.
 \bibitem{Yang} Yang, B: {\em On a problem of Yau regarding a higher dimensional generalization of the Cohn–Vossen inequality}, {\sc Math. Ann.},  {\bf 355} (2013), 765--781.
\bibitem{Y} S.-T. Yau: \textit{Open problems in geometry}, {\sc Lectures on Differential Geometry}, \textbf{1} (1994), 365–404.
 \bibitem{Yau2} S.-T. Yau: {\em Open problems in geometry. Chern-a great geometer of the twentieth century}, pp. 275–319. International Press, Hong Kong (1992)

\end{thebibliography}
\end{document}